\documentclass{article}

\usepackage{hyperref} 
\usepackage[numbers,compress]{natbib}
\usepackage{fourier}
\usepackage{amsmath,amssymb}
\usepackage{physics}
\usepackage{braket}
\usepackage{authblk}
\usepackage{mathtools}
\usepackage{algorithm}
\usepackage{algcompatible}
\usepackage{enumerate}

\usepackage{amsthm}
\theoremstyle{plain}
\newtheorem{thm}{Theorem}[section]
\newtheorem{lem}[thm]{Lemma}

\newtheorem{cor}[thm]{Corollary}

\theoremstyle{definition}

\newtheorem{assumpt}{Assumption}[section]
\theoremstyle{remark}
\newtheorem*{rem}{Remark}

\newcommand{\frec}{f^{\text{rec}}}

\DeclareMathOperator{\dom}{dom}
\DeclareMathOperator{\inte}{int}

\DeclareMathOperator*{\argmin}{arg \, min}

\begin{document}

\title{Two Polyak-Type Step Sizes for Mirror Descent}

\author[1,2]{Jun-Kai~You}
\author[1,3,4]{Yen-Huan~Li}

\affil[1]{Department of Computer Science and Information Engineering, National Taiwan University}
\affil[2]{MediaTek Research}
\affil[3]{Department of Mathematics, National Taiwan University}
\affil[4]{Center for Quantum Science and Engineering, National Taiwan University}

\maketitle

\begin{abstract}
We propose two Polyak-type step sizes for mirror descent and prove their convergences for minimizing convex locally Lipschitz functions. 
Both step sizes, unlike the original Polyak step size, do not need the optimal value of the objective function. 
\end{abstract}

\section{Introduction} \label{sec_intro}

Throughout this paper, consider the optimization problem
\begin{equation}
x^\star \in \argmin_{x \in \mathcal{X}} f ( x ) , \tag{P} \label{eq_problem}
\end{equation}
for some proper closed convex function $f$ and non-empty closed convex set $\mathcal{X} \subseteq \mathbb{R}^d$. 
Denote by $g ( x )$ a subgradient of $f$ at $x$. 
When $\mathcal{X} = \mathbb{R}^d$, \citet{Polyak1969} proposed the subgradient method in Algorithm \ref{alg_polyak}. 
\begin{algorithm}[t]
\caption{Subgradient method with the Polyak step size}
\label{alg_polyak}
\begin{algorithmic}[1]
\REQUIRE $x_1 \in \mathbb{R}^d$. 
\FORALL{$k \in \mathbb{N}$}
\STATE $\eta_k \leftarrow \frac{f ( x_k ) - f^\star }{\norm{ g ( x_k ) }_{2}^2 }$, where $\norm{ \cdot }_2$ denotes the $\ell_2$-norm. 
\STATE $x_{k + 1} = x_k - \eta_k g ( x_k )$. 
\ENDFOR
\end{algorithmic}
\end{algorithm}
The specific choice of $\eta_k$ in Algorithm \ref{alg_polyak} is called the Polyak step size. 
The Polyak step size avoids evaluating the Lipschitz nor smoothness parameters of $f$, sometimes a difficult task. 
Nevertheless, its requires knowing $f^\star$, which limits its direct applications in practice. 

Remarkably, the convergence rate of the Polyak step size is often satisfactory and can be even optimal. 
\citet{Polyak1969} proved that projected gradient descent with the Polyak step size converges at a linear rate when $f$ is strongly convex and either smooth or Lipschitz. 
Later, \citet[Chapter 5.3]{Polyak1987} showed that asymptotically, gradient descent with the Polyak step size converges at an $O ( 1 / \sqrt{k} )$ rate, where $k$ denotes the number of iterations, when $f$ is locally Lipschitz. 
\citet{Hazan2019} showed that the Polyak step size yields the optimal iteration complexities achievable by gradient descent when $f$ is Lipschitz, Lipschitz and strongly convex, smooth, or smooth and strongly convex. 
\citet{Loizou2021} proposed a ``stochastic Polyak step size'' for stochastic gradient descent when $f$ is a finite sum and studied the iteration complexity under standard smoothness and strong convexity conditions. 
The stochastic Polyak step size was later extended by \citet{DOrazio2021} for stochastic mirror descent. 
\citet{Ren2022} studied the convergence of gradient descent with the Polyak step size under a generalized smoothness and generalized {\L}ojasiewicz condition. 


In this paper, we study the Polyak step size in the style of \citet{Polyak1987}. 
We do not aim to provide an iteration complexity bound under restrictive conditions on the objective function $f$, such as Lipschitzness and smoothness. 
Instead, we aim to guarantee asymptotic convergence for a very large class of $f$. 
In particular, we propose two algorithms based on mirror descent with Polyak-type step sizes and prove their convergences when $f$ is locally Lipschitz. 
Our motivation is twofold. 
\begin{enumerate}
\item There are several applications that violate the Lipschitz and smoothness assumptions, such as portfolio selection \citep{MacLean2012}, Poisson inverse problems \citep{Bertero2009,Hohage2016}, quantum state tomography \citep{Hradil1997,Paris2004}, and minimization of quantum R\'{e}nyi divergences \citep{You2022}. 
\item Though \citet{Polyak1987} has proved convergence of gradient descent with the Polyak step size, it is desirable to generalize the result for mirror descent. 
First, \citet{Polyak1987} does not consider the constrained optimization case. 
Second, projected gradient descent can generate infeasible iterates, causing the algorithm to ``stall'' before approaching the minimizer \citep{Knee2018,You2022}. 
\end{enumerate}
The interested reader is referred to Appendix \ref{appendix_problems} for the details. 
Moreover, unlike the original Polyak step size studied by \citet{Polyak1987} and its mirror descent extension \citet{DOrazio2021}, the two algorithms we propose do not need the optimal value $f^\star$. 

There have been several variants of the Polyak stpe size that do not need $f^\star$ either; 
see, e.g., the discussions by \citet{Polyak1969,Brannlund2001} and a recent solution by \citet{Hazan2019}. 
Among existing works, the most relevant to this paper are \citet{Nedic2001} and \citet{Goffin1999}. 
In particular, the two algorithms we propose are generalizations of the ``first adjustment'' considered by \citet{Nedic2001} and the subgradient level method analyzed by \citet{Goffin1999}\footnote{According to \citet{Goffin1999}, the algorithm was proposed by Br\"{a}nnlund in his PhD thesis, but we cannot find the PhD thesis. Therefore, we cite \citet{Goffin1999} instead of Br\"{a}nnlund's PhD thesis.}, respectively.

Our analyses differ significantly from those by \citet{Nedic2001} and by \citet{Goffin1999}. 
In particular, the non-Euclidean nature of mirror descent and the fact that the mirror map may not be defined on the boundary of $\mathcal{X}$ render the proof strategies of \citet{Nedic2001} and \citet{Goffin1999} not directly applicable. 
The measure of the ``traveling distance'' in the subgradient level method is also slightly modified to fit in the mirror descent framework; 
see Section \ref{sec_second_alg} for a discussion. 

The first algorithm we propose had appeared in our recent paper \cite{You2022}. 
That paper focuses on an application in quantum information theory and is targeted at information theorists. 
We isolate the optimization theory part and present it in a slightly more general form in Section \ref{sec_first}. 

\section{Problem Formulation and Useful Facts}

\subsection{Problem Formulation}

We consider solving the optimization problem \eqref{eq_problem} by mirror descent. 
Let $h$ be a convex function and $D_h$ the associated Bregman divergence, i.e., 
\[
D_h ( x, y ) \coloneqq h ( x ) - h ( y ) - \braket{ \nabla h ( y ), x - y } , \quad \forall ( x, y ) \in \dom h \times \dom \nabla h . 
\]
Let $g ( x )$ be a subgradient of $f$ at $x$. 
Define $T ( x; \eta )$ as a solution to the following minimization problem
\begin{equation}
\min_{y \in \mathcal{X}} \braket{ g ( x ), y - x } + \frac{D_h ( y, x )}{\eta} \label{eq_MD}
\end{equation}
for any \emph{step size} $\eta > 0$. 
Mirror descent iterates by iteratively applying the mapping $T$ with possibly different step sizes. 
The two ``algorithms'' we propose are indeed mirror descent with different step size selection rules. 
For convenience, we present the step size selection rule and mirror descent step \eqref{eq_MD} together and call them ``algorithms.''

\subsection{Assumptions}

We make the following assumptions in the rest of this paper. 
The first assumption is standard and ensures that $f$ is continuous around the minimizer \citep[Corollary 8.39]{Bauschke2017a}. 

\begin{assumpt} \label{assump_finite}
The optimal value $f^\star$ is finite. 
\end{assumpt}

The following two assumptions ensure that mirror descent with $D_h$ is well-defined for solving \eqref{eq_problem}. 
In particular, Assumption \ref{assump_legendre} resolves the ``domain consistency'' issue \citep{Bauschke1997}; 
Assumption \ref{assump_subdifferential} ensures that $g ( x )$ is well defined at any $x \in \mathcal{X} \cap \inte \, \dom h$ \citep[Corollary 16.18]{Bauschke2017}. 
See, e.g., the discussion in \citet{Bauschke2017}. 

\begin{assumpt} \label{assump_legendre}
The function $h$ is Legendre and the set $\mathcal{X} \cap \dom f$ is contained in the closure of $\dom h$. 
\end{assumpt}

\begin{assumpt} \label{assump_subdifferential}
The relative interior of $f$ contains $\mathcal{X} \cap \inte \, \dom h$. 
\end{assumpt}

The following assumption is standard for analyzing mirror descent-type methods; 
see, e.g., \citet{Juditsky2012,Juditsky2012a}. 

\begin{assumpt} \label{assump_strong_convexity}
The function $h$ is strongly convex with respect to a norm $\norm{ \cdot }$ (not necessarily $\ell_2$) on $\mathcal{X} \cap \inte \dom h$; 
that is, 
\[
D_h ( x, y ) \geq \frac{1}{2} \norm{ y - x } ^ 2, \quad \forall x, y \in \mathcal{X} \cap \inte \dom h . 
\]
\end{assumpt}

We will write $\norm{ \cdot }_*$ for the norm dual to $\norm{ \cdot }$. 
The last assumption, local boundedness of the subgradients, is key to our analyses. 
This assumption is also exploited by, e.g., \citet{Polyak1987} and \citet{Goffin1999}. 

\begin{assumpt} \label{assump_bounded_gradient}
The mapping $g ( \cdot )$ is bounded on any compact subset of $\mathcal{X} \cap \inte \, \dom h$. 
\end{assumpt}

One may equivalently assume that $f$ is Lipschitz on any compact subset of $\mathcal{X} \cap \inte \, \dom h$. 

\subsection{Useful Facts}

The following results, which will be used in our analyses, are perhaps familiar to experts. 
We present them for the convenience of the reader. 

The following theorem \citep[Theorem 3.8(i)]{Bauschke1997} helps verify that the iterates all lie in the interior of $\dom h$. 

\begin{thm} \label{thm_bauschke_borwein}
Suppose that $h$ is Legendre. 
Let $x \in \inte \dom h$ and $( x_k )_{k \in \mathbb{N}}$ be a sequence in $\inte \dom h$. 
If $D_h ( x, x_k ) \to \infty$, then $( x_k )_{k \in \mathbb{N}}$ converges to a point on the boundary of $\dom h$. 
\end{thm}

\begin{cor} \label{cor_bauschke_borwein}
Suppose that $h$ is Legendre. 
Let $x \in \inte \, \dom h$ and $\set{ x_k }_{k \in \mathbb{N}} \subset \inte \, \dom h$. 
If $D_h ( x, x_k ) \leq c$ for some $c > 0$ for all $k \in \mathbb{N}$, then the closure of $\set{ x_k }_{k \in \mathbb{N}}$ lies in $\inte \, \dom h$. 
\end{cor}

Let $x^\star$ be the minimizer. 
In our analyses, it is desirable to set $x$ to be $x^\star$ in Theorem \ref{thm_bauschke_borwein}. 
Nevertheless, it can happen that $x^\star$ does not belong to\footnote{Consider, for example, minimizing the function $f ( x, y ) \coloneqq x$ on the the set 
\[ \set{ ( x, y ) \in \mathbb{R}^2 : x \geq 0, y \geq 0, x + y = 1 }
\] 
by entropic mirror descent. Then, $h$ is the engative Shannon entropy and $\inte \, \dom h$ is the interior of the positive orthant. Obviously, the minimizer is $( 1, 0 )$ and does not belong to $\inte \, \dom h$} $\inte\, \dom h$. 
The issue can be easily circumvented as $f$ is closed. 

\begin{lem} \label{lem_continuity}
Let $x^\star$ be a minimizer of $f$ on $\mathcal{X}$. 
For any $\varepsilon > 0$, there exists some $x_\varepsilon \in \mathcal{X} \cap \dom h$ such that $f ( x_\varepsilon ) \leq f^\star + \varepsilon$. 
\end{lem}

\begin{proof}
If $x^\star \in \mathcal{X} \cap \inte \, \dom h$, then we can simply choose $x_\varepsilon = x^\star$. 
Otherwise, by Assumption \ref{assump_subdifferential}, there is some $y$ in the intersection of $\mathcal{X} \cap \inte \, \dom h$ and the relative interior of $\dom f$. 
Define 
\[
\varphi ( t ) \coloneqq f ( x^\star + t ( y - x ^ \star ) ), \quad \forall t \in \mathbb{R} . 
\]
It is easily checked that $\varphi$ is also proper closed convex. 
Then, $\varphi$ is continuous on the closure of its domain \citep[Corollary 9.15]{Bauschke2017}. 
Therefore, we can choose $x_\varepsilon = x^\star + t ( y - x^\star )$ for some $t$ small enough. 
\end{proof}

The following is an intermediate result in the standard analysis of mirror descent \citep{Juditsky2012a} and is seldom explicitly stated as a lemma. 
We provide its proof for completeness. 

\begin{lem} \label{lem_standard_result}
Let $x \in \mathcal{X} \cap \inte \, \dom h$. 
Let $x_+ = T ( x; \eta )$ for some $\eta > 0$. 
Then, $x_+ \in \mathcal{X} \cap \inte \, \dom h$ and 
\[
\eta \left( f ( x ) - f ( y ) \right) \leq D_h ( y, x ) - D_h ( y, x_+ ) + \frac{\eta ^ 2 \norm{ g ( x ) }_*^2}{2} . 
\]
\end{lem}

\begin{proof}
We write 
\begin{align*}
\eta_k \left( f ( x ) - f ( y ) \right) & \leq \eta_k \braket{ g ( x ), x - y } \\
& = \eta_k \braket{ g ( x ), x_+ - x } + \eta_k \braket{ g ( x ), x - x_+ } \\
& \leq D_h ( y, x ) - D_h ( y, x_+ ) - D_h ( x_+, x ) + \eta_k \braket{ g ( x ), x - x_+ } \\
& \leq D_h ( y, x ) - D_h ( y, x_+ ) - \frac{\norm{ x - x_+ }^2}{2} + \eta \braket{ g ( x ), x - x_+ } \\
& \leq D_h ( y, x ) - D_h ( y, x_+ ) - \frac{\norm{ x - x_+ }^2}{2} + \eta \norm{ g( x ) }_* \norm{ x - x_+ } . 
\end{align*}
The first inequality follows from the convexity of $f$; 
the second follows from the Bregman proximal inequality \citep{Teboulle2018}; 
the third follows from Assumption \ref{assump_strong_convexity}. 
It remains to maximize the right-hand side with respect to $\norm{ x - x_+ }$. 
\end{proof}

\section{First Algorithm} \label{sec_first}

\subsection{Algorithm and Convergence Guarantee}

Algorithm \ref{alg_first} presents the first algorithm we propose. 
The algorithm generalizes the ``first adjustment'' of \citet{Nedic2001} by replacing the $\ell_2$-norm with a general norm. 
Compared to Algorithm \ref{alg_polyak}, the only difference lies in that $f^\star$ is replaced by a sequence of its estimates $( \hat{f}_k )_{k \in \mathbb{N}}$. 

\begin{algorithm}[t]
\caption{First Algorithm} 
\label{alg_first}
\begin{algorithmic}[1]
\REQUIRE $\delta_1 \geq \delta > 0$, $\beta < 1$, $\gamma \geq 1$, $c > 1 / 2$, and $x_1 \in \mathcal{X} \cap \inte \dom h$. 
\FORALL{$k \in \mathbb{N}$}
\IF{$g(x_k) = 0$}
\STATE Return $x_k$ as a minimizer and terminate. 
\ENDIF
\STATE $\hat{f}_k \leftarrow \min_{1 \leq \kappa \leq k} f ( x_\kappa ) - \delta_k$. 
\STATE $\eta_k \leftarrow \frac{f(x_k) - \hat{f}_k}{c \norm{ g ( x_k ) }_*^2}$. 
\STATE $x_{k + 1} \leftarrow T ( x_k ; \eta_k )$. 
\IF{$f ( x_{k + 1} ) \leq \hat{f}_k$}
\STATE $\delta_{k + 1} = \gamma \delta_k$. 
\ELSE
\STATE $\delta_{k + 1} = \max \set{ \beta \delta_k, \delta }$. 
\ENDIF
\ENDFOR
\end{algorithmic}
\end{algorithm}

The following theorem guarantees that Algorithm \ref{alg_first} asymptotically converges to an approximate solution to \eqref{eq_problem}. 

\begin{thm} \label{thm_first}
If Assumptions 1--5 hold, then Algorithm \ref{alg_first} satisfies
\[
\inf_{k \in \mathbb{N}} f ( x_k ) \leq f ^ \star + \delta . 
\]
\end{thm}

Algorithm \ref{alg_first} requires deciding the error tolerance $\delta$ in advance and does not guarantee convergence to the exact minimum. 
Our second algorithm fixes the weaknesses. 

\subsection{Proof of Theorem \ref{thm_first}} 

Suppose, for contradiction, that 
\begin{equation}
\inf_{k \in \mathbb{N}} f ( x_k ) \geq f^\star + \delta + \varepsilon \label{eq_contradiction_1}
\end{equation}
for some $\varepsilon > 0$. 

\begin{lem} \label{lem_asymptotic_delta}
There is some $\tilde{k} \in \mathbb{N}$ such that $\delta_k = \delta$ for all $k \geq \tilde{k}$. 
\end{lem}

\begin{proof}
Suppose that $f ( x_{k + 1} ) \leq \hat{f}_k$ holds for infinitely many $k$'s. 
Then, since $\delta_k \geq \delta$ for all $k \in \mathbb{N}$, we have $f ( x_k ) \to - \infty$, violating Assumption \ref{assump_finite}. 
Therefore, $f ( x_{k + 1} ) \leq \hat{f}_k$ can only hold for finitely many $k$'s. 
This implies that Line 11 of Algorithm \ref{alg_first} must be executed infinitely many times. 
The lemma follows. 
\end{proof}

By the definition of $\hat{f}_k$, \eqref{eq_contradiction_1}, and Lemma \ref{lem_asymptotic_delta}, we have 
\[
\hat{f}_k \geq \inf_{k \in \mathbb{N}} f ( x_k ) - \delta_k \geq f^\star + \delta + \varepsilon - \delta = f^\star + \varepsilon , \quad \forall k \geq \tilde{k} . 
\]
By Lemma \ref{lem_continuity}, there is some $\tilde{x} \in \mathcal{X} \cap \inte \, \dom h$ such that 
\begin{equation}
\hat{f}_k \geq f ( \tilde{x} ) , \quad \forall k \geq \tilde{k} . \label{eq_tilde_x_1}
\end{equation}
Then, we write 
\begin{align}
D_h ( \tilde{x}, x_{k + 1} ) & \leq D_h ( \tilde{x}, x_k ) - \eta_k \left( f ( x_k ) - f ( \tilde{x} ) \right) + \frac{\eta_k ^ 2 \norm{ g ( x_k ) }_*^2 }{2} \notag \\
& \leq D_h ( \tilde{x}, x_k ) - \eta_k \left( f ( x_k ) - \hat{f}_k \right) + \frac{\eta_k ^ 2 \norm{ g ( x_k ) }_*^2 }{2} \notag \\
& = D_h ( \tilde{x}, x_k ) - \frac{ \left( f ( x_k ) - \hat{f}_k \right)^2 }{ c \norm{ g ( x_k ) }_*^2 } + \frac{\left( f ( x_k ) - \hat{f}_k \right)^2 }{2 c^ 2 \norm{ g ( x_k ) }_*^2} \notag \\
& = D_h ( \tilde{x}, x_k ) - \frac{1}{c} \left( 1 - \frac{1}{2 c} \right) \left( \frac{ f ( x_k ) - \hat{f}_k }{ \norm{ g ( x_k ) }_* } \right)^2 , \quad \forall k \geq \tilde{k} . \label{eq_divergence_bound}
\end{align}
In the above, the first inequality follows from Lemma \ref{lem_standard_result}; 
the second follows from \eqref{eq_tilde_x_1}; 
the third line follows from the definition of $\eta_k$. 

By \eqref{eq_divergence_bound}, the set $\set{ D_h ( \tilde{x}, x_k ) }_{k \in \mathbb{N}}$ is bounded. 
By the strong convexity of $h$ (Assumption \ref{assump_strong_convexity}), the set of iterates $\set{ x_k }_{k \in \mathbb{N}}$ is also bounded in $\mathcal{X} \cap \inte \, \dom h$. 
Hence, by Corollary \ref{cor_bauschke_borwein}, the closure of the set of iterates $\set{ x_k }_{k \in \mathbb{N}}$ is bounded in $\mathcal{X} \cap \inte\, \dom h$. 
Then, there is some $G > 0$ such that $\norm{ g ( x_k ) }_* \leq G$ for all $k \in \mathbb{N}$. 
The inequality \eqref{eq_divergence_bound} implies 
\[
D_h ( \tilde{x}, x_{k + 1} ) \leq D_h ( \tilde{x}, x_k ) - \frac{1}{c G^2} \left( 1 - \frac{1}{2 c} \right) \left( f ( x_k ) - \hat{f}_k \right)^2 , \quad \forall k \geq \tilde{k}
\]
A telescopic sum gives 
\[
\frac{1}{c G^2} \left( 1 - \frac{1}{2 c} \right) \sum_{k = \tilde{k}}^\infty \left( f ( x_k ) - \hat{f}_k \right)^2 \leq D_h ( \tilde{x}, x_{\tilde{k}} ) < \infty , 
\]
showing that $f ( x_k ) - \hat{f}_k \to 0$. 
Nevertheless, Lemma \ref{lem_asymptotic_delta} implies
\[
f ( x_k ) - \hat{f}_k = f ( x_k ) - \min_{\kappa \leq k} f ( \kappa ) + \delta \geq \delta , \quad \forall k \geq \tilde{k} , 
\]
a contradiction. 
The theorem follows. 

\section{Second Algorithm}

\subsection{Algorithm and Convergence Guarantee} \label{sec_second_alg}

Algorithm \ref{alg_second} presents the second algorithm we propose. 
The algorithm generalizes the subgradient level method analyzed by \citet{Goffin1999}. 

\begin{algorithm}[t]
\caption{Second Algorithm}
\label{alg_second}

\begin{algorithmic}[1]
\REQUIRE $\delta_1 > 0$, $\sigma_1 = 0$, $l = 1$, $k ( 1 ) = 1$, $B > 0$, $c > 1 / 2$, and $x_1 \in \mathcal{X} \cap \inte \dom h$. 
\FORALL{$k \in \mathbb{N}$}
\IF{$g(x_k) = 0$}
\STATE Return $x_k$ as a minimizer and terminate. 
\ENDIF
\STATE $\frec_k \leftarrow \min_{1 \leq \kappa \leq k} f ( x_\kappa )$
\IF{$f( x_k ) \leq \frec_{k ( l )} - ( 1 / 2 ) \delta_l$}	
\STATE $k ( l + 1 ) \leftarrow k$
\STATE $\sigma_k \leftarrow 0$
\STATE $\delta_{l + 1} \leftarrow \delta_l$
\STATE $l \leftarrow l + 1$
\ELSIF{$\sigma_k > B$}
\STATE $k ( l + 1 ) \leftarrow k$
\STATE $\sigma_k \leftarrow 0$
\STATE $\delta_{l + 1} \leftarrow ( 1 / 2 ) \delta_l$
\STATE $l \leftarrow l + 1$
\ENDIF
\STATE $\hat{f}_k \leftarrow \frec_{k (l)} - \delta_l$
\STATE $\eta_k \leftarrow \frac{ f ( x_k ) - \hat{f}_k }{ c \norm{ g ( x_k ) }_*^2 }$. 
\STATE $x_{k + 1} \leftarrow T ( x_k; \eta_k )$. 
\STATE $\sigma_{k + 1} \leftarrow \sigma_k + c \eta_k \norm{ g ( x_k ) }_*$. 
\ENDFOR
\end{algorithmic}

\end{algorithm}

Algorithm \ref{alg_second} is a direct adaptation of the subgradient level method to the mirror descent setup, except for Line 20. 
Line 1--17 are exactly the same as the corresponding parts in the subgradient level method. 
Line 18 generalizes the Polyak-type step size with a general norm, as in Algorithm \ref{alg_first}. 
Line 19 replaces a gradient descent step with a mirror descent step. 
Line 20 is the major difference. 
In the original subgradient level method, $\sigma_k$ is the sum of the Euclidean distances between consecutive iterates; 
in Line 20, $\sigma_k$ becomes the sum of the magnitudes of the gradients scaled by $c \eta_k$. 
The two definitions of $\sigma_k$ coincide, up to a scaling factor $c$, in the unconstrained optimization setup \citet{Goffin1999} considered, but are obviously different in general. 

\begin{thm} \label{thm_second}
If Assumptions 1--5 hold, then Algorithm \ref{alg_second} satisfies
\[
\inf_{k \in \mathbb{N}} f ( x_k ) = f ^\star . 
\]
\end{thm}

\subsection{Proof of Theorem \ref{thm_second}}

%

We will prove by contradiction: 
If $\inf_{k \in \mathbb{N}} f ( x_k ) > f^\star$, then Lemma \ref{lem_contradiction_1} and Lemma \ref{lem_contradiction_2} below show that $\delta_l \to 0$ and $\delta_l$ bounded away from zero, respectively.
Therefore, the desired convergence must hold. 

\begin{lem} \label{lem_contradiction_1}
If $\inf_{k \in \mathbb{N}} f ( x_k ) > f^\star$, then $l \to \infty$ and $\delta_l \to 0$. 
\end{lem}

\begin{proof}
By Lemma \ref{lem_continuity}, there is some point $\tilde{x} \in \mathcal{X} \cap \inte \, \dom h$ near $x^\star$ such that 
\begin{equation}
\inf_{k \in \mathbb{N}} f ( x_k ) \geq f ( \tilde{x} ) . \label{eq_tilde_x_2}
\end{equation}
Now we prove $l \to \infty$ by contradiction. 
Suppose that $l$ is always bounded above by some $\overline{l} \in \mathbb{N}$. 
Define $s_k \coloneqq c \eta_k \norm{ g ( x_k ) }_*$. 
By Line 12--15 in Algorithm \ref{alg_second}, we have 
\begin{equation}
\delta_l \geq \underline{\delta} , \quad \forall 1 \leq l \leq \overline{l} , \label{eq_bounded_delta}
\end{equation}
for some $\underline{\delta} > 0$, and there is some $\tilde{k} \in \mathbb{N}$ such that 
\begin{equation}
\sum_{k = \tilde{k} + 1}^\infty s_k \leq B . \label{eq_B}
\end{equation}
Iteratively applying Lemma \ref{lem_standard_result}, we write 
\begin{align*}
D_h ( \tilde{x}, x_{k + 1} ) & \leq D_h ( \tilde{x}, x_k ) - \eta_k \left( f ( x_k ) - f ( \tilde{x} ) \right) + \frac{s_k^2}{2 c^2} \\
& \leq D_h ( \tilde{x}, x_{\tilde{k} + 1} ) - \sum_{\kappa = \tilde{k} + 1}^k \eta_\kappa \left( f ( x_\kappa ) - f ( \tilde{x} ) \right) + \frac{1}{2 c^2} \sum_{\kappa = \tilde{k} + 1}^k s_\kappa ^ 2 \\
& \leq D_h ( \tilde{x}, x_{\tilde{k} + 1} ) + \frac{B^2}{2 c^2} , \quad \forall k > \tilde{k} , 
\end{align*}
where the last inequality follows from \eqref{eq_tilde_x_2} and the fact that $\sum_k s_k^2 \leq \left( \sum_k s_k \right)^2$. 
Therefore, $\set{ D_h ( \tilde{x}, x_k ) }_{k \in \mathbb{N}}$ is bounded. 
By the strong convexity of $h$ (Assumption \ref{assump_strong_convexity}), the set $\set{ x_k }_{k \in \mathbb{N}}$ is also bounded. 
Corollary \ref{cor_bauschke_borwein} then ensures that the closure of $\set{ x_k }_{k \in \mathbb{N}}$ is bounded on the interior of $\dom h$. 
By Assumption \ref{assump_bounded_gradient}, we have $\norm{ g ( x_k ) }_* \leq G$ for some $G > 0$ for all $k \in \mathbb{N}$. 

The inequality \eqref{eq_B} also implies that $\eta_k \norm{ g ( x_k ) }_* \to 0$. 
By the definition of $\eta_k$, we write 
\begin{align*}
\eta_k \norm{ g ( x_k ) }_* & = \frac{f ( x_k ) - \min_{\kappa \leq k} f ( x_\kappa ) + \delta_l}{c \norm{ \nabla f ( x_k ) }_*} \\
& \geq \frac{f ( x_k ) - \min_{\kappa \leq k} f ( x_\kappa ) + \delta_l}{c G} , 
\end{align*}
showing that 
\[
\lim_{k \to \infty} f ( x_k ) - \min_{\kappa \leq k} f ( x_\kappa ) + \delta_l = 0. 
\]
Then, by \eqref{eq_bounded_delta}, for large enough $k$ we have
\[
f ( x_k ) - \min_{\kappa \leq k} f ( x_\kappa ) + \underline{\delta} \leq f ( x_k ) - \min_{\kappa \leq k} f ( x_\kappa ) + \delta_l \leq \frac{1}{2} \underline{\delta} ,  
\]
contradicting the fact that $f ( x_k ) \geq \min_{\kappa \leq k} f ( x_\kappa )$. 

Now we have proved that $l \to \infty$. 
By the monotone convergence theorem, $\delta_l \to \delta_\infty$ for some $\delta_\infty \geq 0$. 
Suppose that $\delta_\infty > 0$. 
Then, Line 11--15 cannot be executed infinitely many times; 
this together with $l \to \infty$ implies that Line 6--10 is executed infinitely many times.
Then, $f ( x_k ) \to \infty$, violating Assumption \ref{assump_finite}. 
Therefore, we have $\delta_\infty = 0$. 
\end{proof}

\begin{rem}
The lemma is inspired by the analysis of \citet[Lemma 2]{Goffin1999}. 
If we can choose $\tilde{x} = x^\star$ in the proof above, which is the case in the analysis of \citet{Goffin1999}, then there is no need to assume $\inf_{k \in \mathbb{N}} f ( x_k ) > f^\star$. 
Nevertheless, $x^\star$ may lie on the boundary of $\dom h$, which hinders us from using Theorem \ref{cor_bauschke_borwein}. 
See the footnote before Lemma \ref{lem_continuity} for an example. 
\end{rem}

\begin{lem} \label{lem_contradiction_2}
If $\inf_{k \in \mathbb{N}} f ( x_k ) > f^\star$, then $\delta_l \to \delta_\infty$ for some $\delta_\infty > 0$. 
\end{lem}

\begin{proof}
Notice that $\hat{f}_k$ and $\frec_{k (l)}$ are decreasing and bounded from below by $f^\star$. 
By the monotone convergence theorem, both $\lim_{k \to \infty} \hat{f}_k$ and $\lim_{k \to \infty} \frec_{k (l)}$ exist. 
Also notice that $\lim_{k \to \infty} \frec_{k (l)} > f^\star$; 
otherwise, we have $\inf_{k \in \mathbb{N}} f ( x_k ) = f^\star$. 
Since $\delta_l \to 0$ by Lemma \ref{lem_contradiction_1}, we write 
\begin{align*}
\lim_{k \to \infty} \hat{f}_k & = \lim_{k \to \infty} \left( \frec_{k (l)} - \delta_l \right) \\
& = \lim_{k \to \infty} \frec_{k (l)} \\
& > f^\star . 
\end{align*}
Similarly as in the proof of Lemma \ref{lem_contradiction_1}, there exist some $\varepsilon > 0$ and $\tilde{x} \in \mathcal{X} \cap \inte \, \dom h$ and $\tilde{k} \in \mathbb{N}$ such that\footnote{Notice that there is a slight abuse of notations: the $\tilde{x}$ and $\tilde{k}$ here are different from those in the proof of Lemma \ref{lem_contradiction_1}.} 
\begin{equation}
f ( x_k ) > f ( \tilde{x} ) \text{~ and ~} \hat{f}_k \geq f ( \tilde{x} ) + \varepsilon , \quad \forall k \geq \tilde{k} . \label{eq_tilde_x}
\end{equation}
Then, we have 
\begin{equation}
s_k \coloneqq c \eta_k \norm{ g ( x_k ) }_* = \frac{ f ( x_k ) - \hat{f}_k }{\norm{ g ( x_k ) }_*} \leq \frac{ f ( x_k ) - f ( \tilde{x} ) - \varepsilon }{ \norm{ g ( x_k ) }_* } , \quad \forall k \geq \tilde{k} . \label{eq_sk}
\end{equation}
We write 
\begin{align}
D_h ( \tilde{x}, x_{k + 1} ) & \leq D_h ( \tilde{x}, x_k ) - \eta_k \left( f ( x_k ) - f ( \tilde{x} ) \right) + \frac{s_k^2}{2 c^2} \notag \\
& = D_h ( \tilde{x}, x_k ) - \frac{s_k \left( f ( x_k ) - f ( \tilde{x} ) \right) }{c \norm{ g ( x_k ) }_*} + \frac{s_k^2}{2 c^2} \notag  \\
& \leq D_h ( \tilde{x}, x_k ) - \frac{s_k \left( f ( x_k ) - f ( \tilde{x} ) \right) }{c \norm{ g ( x_k ) }_*} + \frac{s_k \left( f ( x_k ) - f ( \tilde{x} ) - \varepsilon \right)}{2 c^2 \norm{ g ( x_k ) }_*} \notag \\
& = D_h ( \tilde{x}, x_k ) - \left( c - \frac{1}{2} \right) \frac{s_k \left( f ( x_k ) - f ( \tilde{x} ) \right)}{c^2 \norm{ g ( x_k ) }_*} - \frac{ s_k \varepsilon }{ 2 c^2 \norm{ g ( x_k ) }_*} \notag \\
& \leq D_h ( \tilde{x}, x_k ) - \frac{ s_k \varepsilon }{ 2 c^2 \norm{ g ( x_k ) }_*} , \quad \forall k \geq \tilde{k} .  \label{eq_to_be_telescoped}
\end{align}
In the above, the first inequality follows from Lemma \ref{lem_standard_result}; 
the second inequality follows from \eqref{eq_sk}; 
the last follows from \eqref{eq_tilde_x}. 
Therefore, $D_h ( \tilde{x}, x_k )$ is strictly decreasing.
Similarly as in the proof of Lemma \ref{lem_contradiction_1}, we conclude that $\norm{ g ( x_k ) }_* \leq G$ for some $G > 0$. 
Summing the inequality \eqref{eq_to_be_telescoped} for $k \geq \tilde{k}$, we get
\[
\frac{\varepsilon}{2 c^2 G}\sum_{k = \tilde{k}}^\infty s_k < \infty . 
\]
This implies that $\sigma_k > B$ (Line 11 in Algorithm \ref{alg_second}) can only happen a finite number of times, so $\delta_l$ does not converge to zero. 
\end{proof}

Lemma \ref{lem_contradiction_1} and Lemma \ref{lem_contradiction_2} together provide the desired contradiction. 
The theorem follows. 

\section*{Acknowledgements}

This work was done when J.-K.~You was a research assistant at the Department of Computer Science and Information Engineering, National Taiwan University. 
This work is supported by the Young Scholar Fellowship (Einstein Program) of the National Science and Technology Council of Taiwan under grant numbers MOST 108-2636-E-002-014, MOST 109-2636-E-002-025,  MOST 110-2636-E-002-012, and MOST 111-2636-E-002-019 and by the research project “Pioneering Research in Forefront Quantum Computing, Learning and Engineering” of National Taiwan University under grant number NTU-CC-111L894606. 

\bibliography{list}

\appendix

\section{Optimization Problems Violating Lipschitz/Smoothness Conditions} \label{appendix_problems}

Consider a stock market of $d$ stocks. 
Denote the return rates of the $n$ stocks by an entry-wise non-negative random variable $a$ taking values in $\mathbb{R}^d$. 
Suppose $a$ takes the value $a_i$ with probability $p_i$\footnote{In general, the support of $a$ does not need to be finite. We focus on the finite support case to ease the discussion.}. 
The growth-optimal portfolio selection (aka the Kelly criterion) is given by
\begin{equation}
x^\star \in \argmin_{x \in \Delta} f ( x ) , \quad f ( x ) \coloneqq \sum_{i = 1}^n p_i \left[ - \log \braket{ a_i, x } \right] ,  \label{eq_portfolio}
\end{equation}
where $\Delta$ denotes the probability simplex in $\mathbb{R}^n$. 
It is easily checked that the optimization problem \eqref{eq_portfolio} is convex and the function $f$ is neither Lipschitz nor smooth. 
A simple calculation shows that $f$ is not smooth relative to the negative Shannon entropy \citep{Li2019a}.
Although $f$ is smooth relative to the Burg entropy \citep{Bauschke2017}, the resulting iteration rule lacks a closed-form and can be computationally expensive when the $n$ is very large. 

The maximum-likelihood estimator (MLE) for Poisson inverse problems can be transformed to \eqref{eq_portfolio} after a smart entry-wise scaling \citep{Vardi1993,Ben-Tal2001}. 
The MLE for quantum state tomography takes the form \citep{Hradil1997}
\begin{equation}
\rho^\star \in \argmin_{\rho \in \mathcal{D}} \frac{1}{n} \sum_{i = 1}^n \left[ - \log \tr ( A_i \rho ) \right] , \label{eq_qst}
\end{equation}
where $\mathcal{D}$ denotes the set of Hermitian positive semi-definite matrices of unit traces and $A_i$ are Hermitian positive semi-definite matrices. 
When all matrices involved share the same eigenbasis, \eqref{eq_qst} becomes a vector optimization problem of the form \eqref{eq_portfolio}; 
that is, \eqref{eq_qst} is exactly the non-commutative counterpart of \eqref{eq_portfolio}. 
The expressions of quantum R\'{e}nyi divergences are complicated, so we omit them; 
the interested reader is referred to \citet{You2022}. 

Consider solving \eqref{eq_portfolio} by projected gradient descent. 
Even if we ignore the step size selection issue, the so-called ``stalling issue'' \citep{Knee2018} arises. 
Projection onto the probability simplex typically results in sparse vectors \citep{Condat2016}. 
Since $a_i$ can also be sparse and the product of sparse vectors can be exactly zero, the function $f$ and its gradient can be undefined at a sparse iterate;
then, projected gradient descent ``stalls'', though the iterate may be far from the minimizer. 
\citet{Knee2018} and \citet{You2022} provide related discussions for quantum process tomography and minimizing quantum R\'{e}nyi divergences, respectively. 

%
%

\end{document}